\documentclass[11pt]{article}
\usepackage{amsfonts}
\usepackage{graphics,amssymb,amsthm,amsmath,epsf}
\usepackage{graphicx}
\usepackage{subfigure}

\usepackage[round]{natbib}

\newtheorem{algorithm}{Algorithm}[section]
\newtheorem{theorem}{Theorem}[section]

\numberwithin{equation}{section}
\newtheorem{remark}{Remark}[section]

\begin{document}

\date{}
\title{Inference on the Parameters of the Weibull \\ Distribution Using Records}
\author{A. A. Jafari\thanks{Corresponding: aajafari@yazd.ac.ir} , H. Zakerzadeh \\
{\small Department of Statistics, Yazd University, Yazd, Iran}}
\date{}
\maketitle
\begin{abstract}
The Weibull distribution is a very applicable model for the lifetime data. In this paper,
we have investigated inference on the parameters of Weibull distribution based on record values.
We first propose a simple and exact test and a confidence interval for the shape parameter. Then, in addition to  a generalized confidence interval, a generalized test variable is derived for the scale parameter  when the shape parameter is unknown. The paper presents a simple and exact joint confidence region as well. 
 In all cases, simulation studies show that the proposed approaches are more satisfactory and  reliable than previous methods. All proposed approaches are illustrated using a real example.

\end{abstract}



\section{Introduction}

The Weibull distribution is a well-known distribution that is widely used for lifetime models. It has numerous varieties of shapes and demonstrates considerable flexibility  that enables it to have increasing and decreasing failure rates. Therefore, it is used for many applications for example in hydrology, industrial engineering, weather forecasting and insurance. The Weibull distribution with parameters $\alpha$ and $\beta$, denoted by $W\left(\alpha,\beta\right)$, has a cumulative distribution function (cdf)
\[F\left(x\right)=1-e^{-{\left(\frac{x}{\alpha}\right)}^\beta},\ \ \ x>0,\ \ \ \alpha>0,\ \ \ \beta>0,\]
and probability density function (pdf)
\[f\left(x\right)=\frac{\beta }{{\alpha }^{\beta }}x^{\beta-1}e^{-{\left(\frac{x}{\alpha}\right)}^\beta},\ \ \ \ x>0.\]

The Weibull distribution is a generalization of the exponential distribution and Rayleigh distribution. Also, $Y={\log\left(X\right)}$ has the Gumbel distribution with parameters $b=\frac{1}{\beta }$ and $a={\log \left(\alpha \right)}$, when $X$ has a Weibull distribution with parameters $\alpha $ and
$\beta$.

Let $X_1,X_2,\dots $ be an infinite sequence of independent identically distributed random variables from a same population with the cdf  $F_{\theta }$, where $\theta$ is a parameter. An observation $X_j$ will be called an upper record value (or simply a record) if its value exceeds that of all previous observations. Thus, $X_j$ is a record if $X_j > X_i$ for every $i<j$. An analogous definition deals with lower record values. The record value sequence $\{R_n\}$ is defined by
\[R_n=X_{T_n},\ \ \ n=0,1,2,\dots.\]
where $T_n$ is called the record time of $n$th record and is defined as  $T_n={\min  \{j:X_j>X_{T_{n-1}}\}}$ with $T_0=1$.

Let $R_0,\dots ,R_n$ be the first $n+1$ upper record values from the cdf $F_{\theta }$ and the pdf $f_{\theta }$. Then, the joint distribution of the first $n+1$ record values is given by
\begin{equation}\label{eq.fR}
f_{{\boldsymbol R}}\left({\boldsymbol r}\right)=f_{\theta }\left(r_n\right)\prod^{n-1}_{i=0}{\frac{f_{\theta }\left(r_i\right)}{1-F_{\theta }\left(r_i\right)}},\ \ \ \ r_0<r_1<\dots <r_n,
\end{equation}
where ${\boldsymbol r}=(r_0,r_1,\dots ,r_n)$ and ${\boldsymbol R}=(R_0,R_1,\dots ,R_n)$ \citep[for more details see][]{ar-ba-na-98}.

\cite{chandler-52}
launched a statistical study of the record values, record times and inter record times. Record values and the associated statistics are of interest and importance in the areas of meteorology, sports and economics.
\cite{ahsanullah-95}
and
\cite{ar-ba-na-98}
are two good references about records and their properties.

Some papers considered inference on the Weibull distribution based on record values:
\cite{dallas-82}
discussed some distributional results based on upper record values.
\cite{ba-ch-94}
established some simple recurrence relations satisfied by the single and the product moments,
and derived the BLUE of the scale parameter when the shape parameter is known.
 \cite{chan-98}
provided a conditional method to derive exact intervals for location and scale parameters of location-scale family that  can be used to derive exact intervals for the shape parameter.
\cite{wu-ts-06}
provided some pivotal quantities to test and establish confidence interval of the shape parameter based on the first $n+1$ observed upper record values.
\cite{so-ab-su-06}
derived the Bayes estimates based on record values for the parameters with respect to squared error loss function and LINEX loss function.
\cite{as-ab-11}
proposed joint confidence regions for the parameters.
\cite{te-gu-12}
computed the coefficient of skewness of upper/lower record statistics.
\cite{te-na-13}
derived exact expressions for constructing bias corrected maximum likelihood estimators (MLE's) of the parameters for the Weibull distribution based on upper records.
\cite{go-lo-sa-12} obtained the asymptotic properties
 for the counting process of $\delta$-records among the first
$n$ observations.

In this paper, we consider inference about the parameters of Weibull distribution based on record values. First, we will propose a simple and exact method for constructing confidence interval and testing the hypotheses about the shape parameter $\beta $. Then using the concepts of generalized {\it p}-value and generalized confidence interval,  a generalized approach for inference about the scale parameter $\alpha $ will be derived.
\cite{ts-we-89}
introduced the concept of generalized {\it p}-value, and
\cite{weerahandi-93}
introduced the concept of generalized confidence interval. These approaches have been used successfully to address several complex problems
\cite[see][]{weerahandi-95}
such as
confidence interval for the common mean of several log-normal distributions \citep{be-ja-06-ge},
confidence interval for the mean of Weibull distribution
\citep{kr-li-xi-09},
inference about the stress-strength reliability involving two independent Weibull distributions
\citep{kr-li-10},
and comparing two dependent generalized variances \citep{jafari-12}.

We also present an exact joint confidence region for the parameters. Our simulation studies show that the area of our joint confidence region is smaller than those provided by other existing methods.

The rest of this article is organized as follows: A simple method for inference about shape parameter and a generalized approach for inference about the scale parameter are proposed in Section \ref{sec.inf}.
Furthermore, a simulation study is performed and a real example is proposed in this Section. We also present a joint confidence region for the parameters $\alpha $ and $\beta $ in Section \ref{sec.joint}.

\section{ Inference on the parameters}
\label{sec.inf}
Suppose $R_0, R_1,...,R_n$ are the first $n+1$ upper record values from a Weibull distribution with parameters $\alpha $ and $\beta $. In this section, we consider inference on the parameters $\alpha $ and $\beta$.
From \eqref{eq.fR}, the joint distribution of these record values can be written as
\begin{equation}\label{eq.rfw}
f_{\boldsymbol{R}}\left(\boldsymbol r\right)=\frac{{\beta }^{n+1\ }}{{\alpha }^{\beta (n+1)}}e^{-{\left(\frac{r_n}{\alpha}\right)}^\beta}\prod^n_{i=0}{{r_i}^{\beta-1}}\ \ \ \ \ 0<r_0<r_1<\dots <r_n.
\end{equation}
Therefore, $\left(R_n,\sum^n_{i=0}{{\log  \left(R_i\right)}}\right)$ is sufficient statistic for $\left(\alpha,\beta \right)$. Moreover, it can be easily shown that the  MLE's of the parameters $\alpha $ and $\beta $ are
\begin{equation}\label{eq.mle}
\hat{\beta}=\frac{n+1}{\sum^n_{i=0}{{\log  \left(\frac{R_n}{R_i}\right)}}},\ \ \ \ \ \ \ \hat{\alpha }=\frac{R_n}{{\left(n+1\right)}^{\frac{1}{\hat{\beta }}}}.
\end{equation}

\bigskip
\begin{theorem} \label{thm.1}
Let $R_0,R_1,...,R_n$ be the first $n+1$ upper record values from a Weibull distribution. Then
\begin{enumerate}
  \item[i.] $U=2\beta \sum^n_{i=0}{{\log  \left(\frac{R_n}{R_i}\right) }}$ has a chi-square distribution with $2n$ degrees of freedom.
  \item [ii.] $V=2{\left(\frac{R_n}{\alpha}\right)}^{\beta }$ has a chi-square distribution with $2n+2$ degrees of freedom.
  \item [iii.] $U$ and $V$ are independent.
\end{enumerate}
\end{theorem}

\begin{proof}
\noindent i. Define
\begin{equation}\label{eq.Q}
Q_m=\frac{R_m}{R_{m-1}},\ \ \ m=1,2,\dots,n.
\end{equation}
From
\cite{ar-ba-na-98}
page  20, $Q_m$'s  are independent random variables with
\[P\left(Q_m>q\right)=q^{-\beta m},\ \ \ \ q>1,\]
and
\[2\beta m{\log  (Q_m)}=2\beta m{\log  (\frac{R_m}{R_{m-1}})}\sim {\chi }^2_{(2)}.\]
Therefore,
\begin{eqnarray*}
U&=&2\beta \sum^{n}_{i=0}{{\log(\frac{R_n}{R_i})}}
=2\beta \sum^{n-1}_{i=0}{{\log\left(\frac{R_n}{R_{n-1}}.\frac{R_{n-1}}{R_{n-2}}\dots \frac{R_{i+1}}{R_i}\right)}}\\
&=&2\beta \sum^{n-1}_{i=0}{\sum^n_{m=i+1}{{\log  (\frac{R_m}{R_{m-1}})}}}=2\beta \sum^n_{m=1}{\sum^{m-1}_{i=0}{{\log  (Q_m)}}}=\sum^n_{m=1}{2\beta m{\log(Q_m)}},
\end{eqnarray*}
has a chi-square distribution with $2n$ degrees of freedom.

\noindent ii. Define
\[Y={(\frac{X}{\alpha })}^{\beta },\]
where $X$ has a Weibull distribution with parameters $\alpha $ and $\beta $. Then, $Y$ has an exponential distribution with parameter one. Therefore, we can conclude that $V$ has a chi-square distribution with $2n+2$ degrees of freedom
\citep[see][page  9]{ar-ba-na-98}.

\noindent iii. Let $\beta $ be known. Then,  it can be concluded from \eqref{eq.rfw} that $R_n$  is a complete sufficient statistic for $\alpha$. Also, $Q_m$'s in \eqref{eq.Q} are ancillary statistics. Therefore, $R_n$ and $Q_m$'s are independent, and the proof is completed.
\end{proof}

\subsection{Inference on the shape parameter }
Here, we consider inference on the shape parameter, $\beta$ from a Weibull distribution based on record values, and propose a simple and exact method for constructing a confidence interval and testing the one-sided hypotheses
\begin{equation}\label{eq.H0b1}
H_0:\beta \leq{\beta }_0\ \ \ \ vs.\ \ \ \ \ H_1:\beta >{\beta }_0,
\end{equation}
and the two-sided hypotheses
\begin{equation}\label{eq.H0b2}
H_0:\beta ={\beta }_0\ \ \ \ vs.\ \ \ \ \ H_1:\beta \neq{\beta }_0,
\end{equation}
where ${\beta }_0$ is a specified value.

Based on Theorem \ref{thm.1}, $U=2\beta \sum^n_{i=0}{{\log  \left(\frac{R_n}{R_i}\right)}}$ has a chi-square distribution with $2n$ degrees of freedom. Therefore, a $100\left(1-\gamma\right)\%$  confidence interval for $\beta $ can be obtained as
\begin{equation}\label{eq.cib}
\left(\frac{{\chi}^2_{\left(2n\right),\gamma /2}}{2\sum^n_{i=0}{{\log  \left(\frac{R_n}{R_i}\right)}}}, \frac{{\chi }^2_{\left(2n\right),1-\gamma/2}}{2\sum^n_{i=0}{{\log  \left(\frac{R_n}{R_i}\right)}}}\right),
\end{equation}
where ${\chi }^2_{\left(k\right),\gamma }$ is the $\gamma $th percentile of the chi-square distribution with $k$ degrees of freedom.
Also, for testing the hypotheses in \eqref{eq.H0b1} and \eqref{eq.H0b2}, we can define the test statistic
\[U_0=2{\beta }_0\sum^n_{i=0}{{\log\left(\frac{R_n}{R_i}\right)}},\]
and  the null hypothesis in \eqref{eq.H0b1} is rejected at nominal level $\gamma$ if
\[U_0>{\chi }^2_{\left(2n\right),1-\gamma },\]
and the null hypothesis in \eqref{eq.H0b2} is rejected if
\[U_0<{\chi }^2_{\left(2n\right),\gamma /2}\ \ \ \ \ {\rm or}\ \ \ \ \ U_0>{\chi }^2_{\left(2n\right),1-\gamma /2}.\]

\bigskip
\cite{wu-ts-06}
proposed the random variable
\[W\left(\beta \right)=\frac{\sum^n_{i=0}{R^{\beta }_i\ }}{(n+1){\left(\prod^n_{i=0}{R_i}\right)}^{\frac{\beta }{n+1}}},\]
for inference about the shape parameter, and showed that $W(\beta )$ is an increasing function with respect to $\beta $.  Also, its distribution does not depend on the parameters $\alpha $ and $\beta $. In fact, $W(\beta )$ is distributed as
\[W^*=\frac{\sum^n_{i=0}{R^*_i\ }}{(n+1){\left(\prod^n_{i=0}{R^*_i}\right)}^{\frac{1}{n+1}}},\]
where $R^*_i\ $is the $i$th record from the exponential distribution with parameter one. However, its exact distribution is very complicated, and
\cite{wu-ts-06}
obtained the percentiles of $W\left(\beta \right)$ using Monte Carlo simulation. The confidence limits for $\beta$ are obtained by solving the following equations numerically as
\begin{equation}\label{eq.wu}
W\left(\beta \right)=W^*_{1-\gamma /2},\ \ \ \ \ \ \ \ W\left(\beta \right)=W^*_{\gamma /2},
\end{equation}
where $W^*_{\delta }$ is the $\delta $th percentile of the distribution of $W^*$.

\subsection{Inference on the scale parameter}

Here, we consider inference about the scale parameter, $\alpha $ for a Weibull distribution based on record values, and propose an approach for constructing a confidence interval and testing the one-sided hypotheses
\begin{equation}\label{eq.H0a1}
H_0:\alpha \leq{\alpha }_0\ \ \ \ {\rm vs}.\ \ \ \ \ H_1:\alpha >{\alpha }_0,
\end{equation}
and the two-sided hypotheses
\begin{equation}\label{eq.H0a2}
H_0:\alpha ={\alpha }_0\ \ \ \ {\rm vs}.\ \ \ \ \ H_1:\alpha \neq{\alpha }_0,
\end{equation}
where ${\alpha }_0$ is a specified value.

We did not find any approach in literature for inference about $\alpha$ based on record values when the shape parameter is unknown. Here, we use the concepts of generalized {\it p}-value and generalized confidence interval 
 introduced by
\cite{ts-we-89},
and
\cite{weerahandi-93}, respectively.
In appendix, we briefly review these concepts, and  
 refer  readers to
\cite{weerahandi-95} for more details.

Let
\begin{equation}\label{eq.gpt}
T=r_n{(\frac{2}{V})}^{\frac{2 C_r}{U}}=r_n{(\frac{\alpha }{R_n})}^{\frac{C_r}{\sum^n_{i=0}\log  (\frac{R_n}{R_i})}},
\end{equation}
where $C_r=\sum^n_{i=0}{{\log\left(\frac{r_n}{r_i}\right)}}$, and $r_i$, $i=0,1,\dots, n$ is the observed value of $R_i$, $i=0,1,\dots,n$, and $U$ and $V$ are independent random variables that are defined in Theorem \ref{thm.1}. The observed value of $T$ is $\alpha $, and distribution of $T$ does not depend on unknown parameters $\alpha $ and $\beta $. Therefore, $T$ is a generalized pivotal variable for $\alpha $, and can be used to construct a generalized confidence interval for $\alpha $.

Let
\[T^*=T-\alpha =r_n{(\frac{2}{V})}^{\frac{2 C_r}{U}}-\alpha.\]
Then, $T^*$ is a generalized test variable for $\alpha $, because i) the observed value of $T^*$ does not depend on any parameters, ii) the distribution function of $T^*$ is free from nuisance parameters and only depends on the parameter $\alpha $, and iii) the distribution function of $T^*$ is an increasing function with respect to the parameter $\alpha $, and so, the distribution of $T^*$ is stochastically decreasing in $\alpha $. Therefore, the generalized {\it p}-value for testing the hypotheses in \eqref{eq.H0a1} is given as
\begin{equation}\label{eq.pvg1}
p=P\left(T^*<0|H_0\right)=P\left(T<\alpha_0\right),
\end{equation}
and the generalized {\it p}-value for testing the hypotheses in \eqref{eq.H0a2} is given as
\begin{equation}\label{eq.pvg2}
p=2 \ {\min  \left\{P\left(T>\alpha_0\right),P\left(T<\alpha_0\right)\right\}}.
\end{equation}

The generalized confidence interval for $\alpha$ based on $T$, and the generalized {\it p}-values in \eqref{eq.pvg1} and \eqref{eq.pvg2} can be computed using Monte Carlo simulation  \citep{weerahandi-95,be-ja-06-ge} based on the following algorithm:

\begin{algorithm}\label{alg.1}
For given $r_0,r_1,\dots,r_n$,
\begin{enumerate}
\item  Generate $U\sim {\chi }^2_{(2n)}$ and $V\sim {\chi }^2_{(2n+2)}$.
\item  Compute $T$ in \eqref{eq.gpt}.
\item  Repeat steps 1 and 2 for a large number times, (say $M=10000$), and obtain the values $T_1,\dots,T_M$.
\item  Set $D_l=1$ if $T_l<\alpha_0$ else $D_l=0$, $l=1,...,M$.
\end{enumerate}

The $100\left(1-\gamma \right)\%$ generalized confidence interval for $\alpha $ is $\left[T_{\left(\gamma /2\right)},T_{\left(1-\gamma /2\right)}\right]$, where $T_{\left(\delta \right)}$ is the $\delta $th percentile of $T_l$'s. Also, the generalized {\it p}-value for testing the one-sided hypotheses in \eqref{eq.pvg1} is obtained by $\frac{1}{M}\sum^M_{l=1}{D_l}$.

\end{algorithm}

\subsection{Real example}
\label{sec.ex}
\cite{roberts-79}
gave monthly and annual maximal of one-hour mean concentration of sulfur dioxide (in pphm) from Long Beach, California, for 1956 to 1974.
\cite{chan-98}
showed that the Weibull distribution is a reasonable model for this data set.
\cite{wu-ts-06}
 also study this data set. The upper record values for the month of October from the data are
\begin{center}
26, 27, 40, 41.
\end{center}

The 95\% confidence interval  for the scale parameter $\alpha $ based on our generalized confidence interval with $M=10000$ is obtained as $(5.4869, 39.9734)$. The 95\% confidence interval for the shape parameter $\beta$  in \eqref{eq.cib} is obtained as $(0.6890, 8.0462)$, and based on Wu and Tseng's method in \eqref{eq.wu} is obtained as $(0.6352,7.7423)$. Also, the generalized p-value equals to 0.0227 for testing the hypotheses in \eqref{eq.H0a1} with $\alpha_0=5$. Therefore, the null hypothesis is rejected.

\begin{table}[ht]
\begin{center}
\caption{ Empirical coverage probabilities and expected lengths of the generalized confidence interval for the parameter $\alpha$  with confidence level 0.95.} \label{tab.a}

\begin{tabular}{|c|c|c|ccccccc|} \hline
&  &              &    \multicolumn{7}{|c|}{$\beta $} \\ \hline
           & $\alpha$  & $n$  & 0.5 & 1.0 & 1.2 & 1.5 & 2.0 & 3.0 & 5.0 \\ \hline
Empirical  & 1.0 & 3 &  0.951  & 0.949 & 0.953 & 0.947 & 0.947 & 0.948 & 0.948 \\
Coverage   &     & 7 &  0.952  & 0.949 & 0.950 & 0.950 & 0.951 & 0.953 & 0.952 \\
           &     & 9 &  0.951  & 0.948 & 0.953 & 0.951 & 0.948 & 0.949 & 0.950 \\
           &     &14 &  0.945  & 0.949 & 0.950 & 0.950 & 0.954 & 0.952 & 0.952 \\ \cline{2-10}
           & 2.0 & 3 &  0.949  & 0.952 & 0.947 & 0.949 & 0.951 & 0.950 & 0.953  \\
           &     & 7 &  0.948  & 0.953 & 0.950 & 0.946 & 0.954 & 0.948 & 0.951 \\
           &     & 9 &  0.952 & 0.948 & 0.953 & 0.950 & 0.952 & 0.953 & 0.954  \\
           &     &14 & 0.950 & 0.946 & 0.949 & 0.951 & 0.951 & 0.952 & 0.955    \\  \hline \hline
 Expected  & 1.0 & 3 &  16.740 & 3.581 & 2.804 & 2.155 & 1.653 & 1.211 & 0.847 \\
Length     &     & 7 &  13.575 & 3.198 & 2.477 & 1.942 & 1.475 & 1.041 & 0.686 \\
           &     & 9 & 13.505 & 3.138 & 2.469 & 1.918 & 1.446 & 1.008 & 0.651 \\
           &     &14 & 13.122 & 3.082 & 2.403 & 1.854 & 1.376 & 0.943 & 0.596 \\ \cline{2-10}
           & 2.0 & 3 & 33.516 & 7.187 & 5.579 & 4.341 & 3.323 & 2.427 & 1.704 \\
           &     & 7 & 27.960 & 6.344 & 4.999 & 3.899 & 2.960 & 2.080 & 1.364 \\
           &     & 9 & 27.342 & 6.304 & 4.935 & 3.831 & 2.890 & 2.016 & 1.302 \\
           &     &14 & 26.626 & 6.129 & 4.779 & 3.705 & 2.757 & 1.886 & 1.191  \\ \hline
\end{tabular}
\end{center}
\end{table}

\subsection{Simulation study}
\label{sec.sim1}
We performed a simulation study in order to evaluate the accuracy of proposed methods for constructing confidence interval for the parameters of Weibull distribution. For this purpose, we generated  $n+1$ record values from a Weibull distribution, and   considered $\alpha =1, 2$. For the simulation with 10000 runs and  different values of the shape parameter $\beta$, the empirical coverage probabilities and expected lengths of the methods with the confidence coefficient 0.95 were obtained.
The results of our generalized confidence interval for inference on $\alpha $ using the algorithm \ref{alg.1}  with $M=10000$  are presented in Table \ref{tab.a}, and the results of our exact method (E) and the Wu method (W) for inference on $\beta $  are given in Table \ref{tab.b}.
We can conclude that

\begin{enumerate}
  \item [i.] The empirical coverage probabilities of all methods are close to the confidence level 0.95.
  \item [ii.] The expected lengths of E and W increase when the parameter $\beta $ increases. Additionally, the expected length of E is smaller than W especially when $\beta$ is large.
  \item [iii.] The expected length of our generalized confidence interval for $\alpha $ decreases when the parameter $\beta $ increases. Moreover, it is very large when $\beta$ is small.

     \item [iv.] The expected lengths of all methods decrease when the number of records increases.

     \item [v.]  The empirical coverage probabilities and expected lengths of W and E do not change when the parameter $\alpha$ changes.

\end{enumerate}

\begin{table}
\begin{center}
\caption{ Empirical coverage probabilities and expected lengths of the methods for constructing confidence interval for the parameter $\beta $ with confidence level 0.95} \label{tab.b}

\begin{tabular}{|c|c|c|c|ccccccc|} \hline
&  &     &         &    \multicolumn{7}{|c|}{$\beta $} \\ \hline
& $\alpha$  & $n$ & Method & 0.5 & 1.0 & 1.2 & 1.5 & 2.0 & 3.0 & 5.0 \\ \hline
Empirical& 1.0 & 3 & W & 0.950 & 0.952 & 0.952 & 0.949 & 0.949 & 0.953 & 0.947 \\
Coverage &     &   & E & 0.949 & 0.953 & 0.953 & 0.950 & 0.948 & 0.953 & 0.946 \\ \cline{3-11}
         &     & 7 & W & 0.951 & 0.949 & 0.950 & 0.948 & 0.949 & 0.949 & 0.953 \\
         &     &   & E & 0.951 & 0.950 & 0.948 & 0.950 & 0.953 & 0.953 & 0.950 \\ \cline{3-11}
         &     & 9 & W & 0.949 & 0.949 & 0.945 & 0.949 & 0.947 & 0.949 & 0.950 \\
         &     &   & E & 0.948 & 0.950 & 0.948 & 0.948 & 0.949 & 0.952 & 0.949 \\ \cline{3-11}
         &     &14 & W & 0.946 & 0.949 & 0.951 & 0.951 & 0.953 & 0.949 & 0.951 \\
         &     &   & E & 0.947 & 0.948 & 0.950 & 0.951 & 0.953 & 0.952 & 0.952 \\ \cline{2-11}

         & 2.0 & 3 & W & 0.954 & 0.955 & 0.948 & 0.950 & 0.950 & 0.952 & 0.949 \\
         &     &   & E & 0.953 & 0.953 & 0.947 & 0.949 & 0.949 & 0.952 & 0.950 \\ \cline{3-11}
         &     & 7 & W & 0.950 & 0.955 & 0.947 & 0.948 & 0.952 & 0.947 & 0.948 \\
         &     &   & E & 0.949 & 0.952 & 0.951 & 0.948 & 0.952 & 0.947 & 0.950 \\ \cline{3-11}
         &     & 9 & W & 0.953 & 0.948 & 0.956 & 0.950 & 0.953 & 0.951 & 0.952 \\
         &     &   & E & 0.952 & 0.948 & 0.951 & 0.951 & 0.953 & 0.951 & 0.953 \\ \cline{3-11}
         &     &14 & W & 0.948 & 0.947 & 0.949 & 0.950 & 0.951 & 0.951 & 0.952 \\
         &     &   & E & 0.950 & 0.947 & 0.949 & 0.949 & 0.953 & 0.951 & 0.955 \\  \hline \hline

Expected & 1.0 & 3 & W & 1.704 & 3.431 & 4.194 & 5.279 & 6.913 & 10.396 & 17.479 \\
Length   &     &   & E & 1.630 & 3.285 & 4.013 & 5.041 & 6.611 & 9.937 & 16.722 \\ \cline{3-11}
         &     & 7 & W & 0.932 & 1.879 & 2.224 & 2.808 & 3.752 & 5.578 & 9.276 \\
         &     &   & E & 0.853 & 1.716 & 2.038 & 2.574 & 3.437 & 5.115 & 8.499 \\ \cline{3-11}
         &     & 9 & W & 0.806 & 1.603 & 1.928 & 2.412 & 3.222 & 4.797 & 8.024 \\
         &     &   & E & 0.730 & 1.450 & 1.748 & 2.185 & 2.928 & 4.352 & 7.267 \\ \cline{3-11}
         &     &14 & W & 0.625 & 1.262 & 1.509 & 1.888 & 2.505 & 3.766 & 6.263 \\
         &     &   & E & 0.558 & 1.125 & 1.343 & 1.685 & 2.236 & 3.353 & 5.590 \\ \cline{2-11}

         & 2.0 & 3 & W & 1.713 & 3.458 & 4.156 & 5.277 & 6.856 & 10.307 & 16.998 \\
         &     &   & E & 1.638 & 3.306 & 3.967 & 5.053 & 6.560 & 9.859 & 16.266 \\ \cline{3-11}
         &     & 7 & W & 0.934 & 1.866 & 2.208 & 2.822 & 3.738 & 5.629 & 9.392 \\
         &     &   & E & 0.854 & 1.710 & 2.026 & 2.589 & 3.419 & 5.151 & 8.581 \\ \cline{3-11}
         &     & 9 & W & 0.808 & 1.600 & 1.923 & 2.418 & 3.189 & 4.798 & 8.032 \\
         &     &   & E & 0.733 & 1.451 & 1.743 & 2.193 & 2.890 & 4.347 & 7.276 \\ \cline{3-11}
         &     &   & W & 0.628 & 1.260 & 1.496 & 1.888 & 2.523 & 3.768 & 6.302 \\
         &     &   & E & 0.560 & 1.124 & 1.338 & 1.688 & 2.250 & 3.366 & 5.624 \\  \hline

\end{tabular}
\end{center}
\end{table}

\section{Joint confidence regions for the parameters}
\label{sec.joint}
Suppose $R_0,R_1,\dots,R_n$ are the first $n+1$ upper record values from a Weibull distribution with parameters $\alpha$ and $\beta$. In this section, we presented a joint confidence region for the parameters $\alpha $ and $\beta $. This is important because it can be used to find confidence bounds for any function of the parameters such as the reliability function $R\left(t\right)=\exp  (-(\frac{t}{\alpha })^{\beta })$. For more  references about the  joint confidence region based on records, reader can see \cite{as-ab-11-confidence,as-ab-11} and \cite{as-ab-kus-11}.

\subsection{ Asgharzadeh and Abdi method}

 \cite{as-ab-11}
 present exact joint confidence regions for the parameters of Weibull distribution based on the record values using the idea presented by
\cite{wu-ts-06}.
The following inequalities determine $100\left(1-\gamma \right)\%$ joint confidence regions for $\alpha $ and $\beta$:
\begin{equation}\label{joint.Aj}
A_j=\left\{ \begin{array}{l}
\dfrac{{\log  \left(\left(\frac{n-j+1}{j}\right)k_1+1\right)}}{{\log  \left(\frac{R_n}{R_{j-1}}\right)}}<\beta <\dfrac{{\log  \left(\left(\frac{n-j+1}{j}\right)k_2+1\right)}}{{\log \left(\frac{R_n}{R_{j-1}}\right)}} \\
R_n{\left(\dfrac{2}{{\chi }^2_{\left(2n+2\right),(1+\sqrt{1-\gamma })/2}}\right)}^{\frac{1}{\beta }}<\alpha <R_n{\left(\dfrac{2}{{\chi }^2_{\left(2n+2\right),(1-\sqrt{1-\gamma })/2}}\right)}^{\frac{1}{\beta }}, \end{array}
\right.
\end{equation}
for $j=1,\dots,n$, where
$$
k_1=F_{\left(2n-2j+2,2j\right),(1-\sqrt{1-\gamma })/2} \ \ \ \ k_2=F_{\left(2n-2j+2,2j\right),(1+\sqrt{1-\gamma })/2},
$$
and $F_{\left(a,b\right),\gamma }$ is the $\gamma $th percentile of the F distribution with $a$  and $b$ degrees of freedom. Note that for each $j$, we have a joint confidence region for $\alpha $ and $\beta $.
\cite{as-ab-11} found that in most cases $A_{ \lfloor \frac{n+1}{5}\rfloor}$ and $A_{{ \lfloor \frac{n+1}{5}+1\rfloor}}$ provide the smallest confidence areas, where $\lfloor x \rfloor$ is the largest integer value smaller than $x$.

\subsection{A new joint confidence region}

From Theorem \ref{thm.1},  $U=2\beta\sum^n_{i=0}{{\log  \left(\frac{R_n}{R_i}\right)}}$ has a chi-square distribution with $2n$ degrees of freedom and $V=2{\left(\frac{R_n}{\alpha }\right)}^{\beta }$ has a chi-square distribution with $2n+2$ degrees of freedom, and $U$ and $V$ are independent. Therefore, an exact joint confidence region for the parameters $\alpha $ and $\beta $ of Weibull distribution based on the record values can be given as
\begin{equation}\label{joint.B}
B=\left\{ \begin{array}{l}
\dfrac{{\chi }^2_{\left(2n\right),(1-\sqrt{1-\gamma })/2}}{2\sum^n_{i=0}{{\log  \left(\frac{R_n}{R_i}\right)}}}<\beta <\dfrac{{\chi }^2_{\left(2n\right),(1+\sqrt{1-\gamma })/2}}{2\sum^n_{i=0}{{\log\left(\frac{R_n}{R_i}\right)}}} \\
R_n{\left(\dfrac{2}{{\chi }^2_{\left(2n+2\right),(1+\sqrt{1-\gamma })/2}}\right)}^{\frac{1}{\beta }}<\alpha <R_n{\left(\dfrac{2}{{\chi }^2_{\left(2n+2\right),(1-\sqrt{1-\gamma })/2}}\right)}^{\frac{1}{\beta }}. \end{array}
\right.
\end{equation}

\bigskip
\begin{remark}
All record values are used in the proposed joint confidence region in \eqref{joint.B} but not in the proposed joint confidence regions in \eqref{joint.Aj}.
\end{remark}

\subsection{ Real example}

Here, we consider the upper record values in the  example  given  in Section \ref{sec.ex}. Therefore, the 95\% joint confidence regions for $\alpha $ and $\beta $ based on
\cite{as-ab-11}
 in \eqref{joint.Aj} are
\begin{eqnarray*}
&&
A_1=\left\{\left(\alpha ,\beta \right): 0.5826<\beta <\ 11.9955,\ \ 41{\left(0.1029\right)}^{\frac{1}{\beta }}<\alpha <41{\left(1.1318\right)}^{\frac{1}{\beta }}\right\}
\\
&&A_2=\left\{\left(\alpha ,\beta \right): 0.1646<\beta <\ 6.4905,\ \ \  41{\left(0.1029\right)}^{\frac{1}{\beta }}<\alpha <41{\left(1.1318\right)}^{\frac{1}{\beta }}\right\}
\\
&&
A_3=\left\{\left(\alpha ,\beta \right): 0.1720<\beta <\ 58.9824,\ \ 41{\left(0.1029\right)}^{\frac{1}{\beta }}<\alpha <41{\left(1.1318\right)}^{\frac{1}{\beta }}\right\}
\end{eqnarray*}
and the 95\% joint confidence region for $\alpha $ and $\beta $ in \eqref{joint.B} is
\[B=\left\{\left(\alpha ,\beta \right): 0.5305<\beta <\ 9.0277,\ \ 41{\left(0.1029\right)}^{\frac{1}{\beta }}<\alpha <41{\left(1.1318\right)}^{\frac{1}{\beta }}\right\}.\]

The plot of all joint confidence regions are given in Figure  \ref{fig.joint}. Also, the area of the joint confidence regions $A_1$, $A_2$, $A_3$, and $B$ are 194.9723, 166.7113, 369.7654, and 172.5757, respectively.

\begin{figure}[ht]
\centering
\includegraphics[scale=0.5]{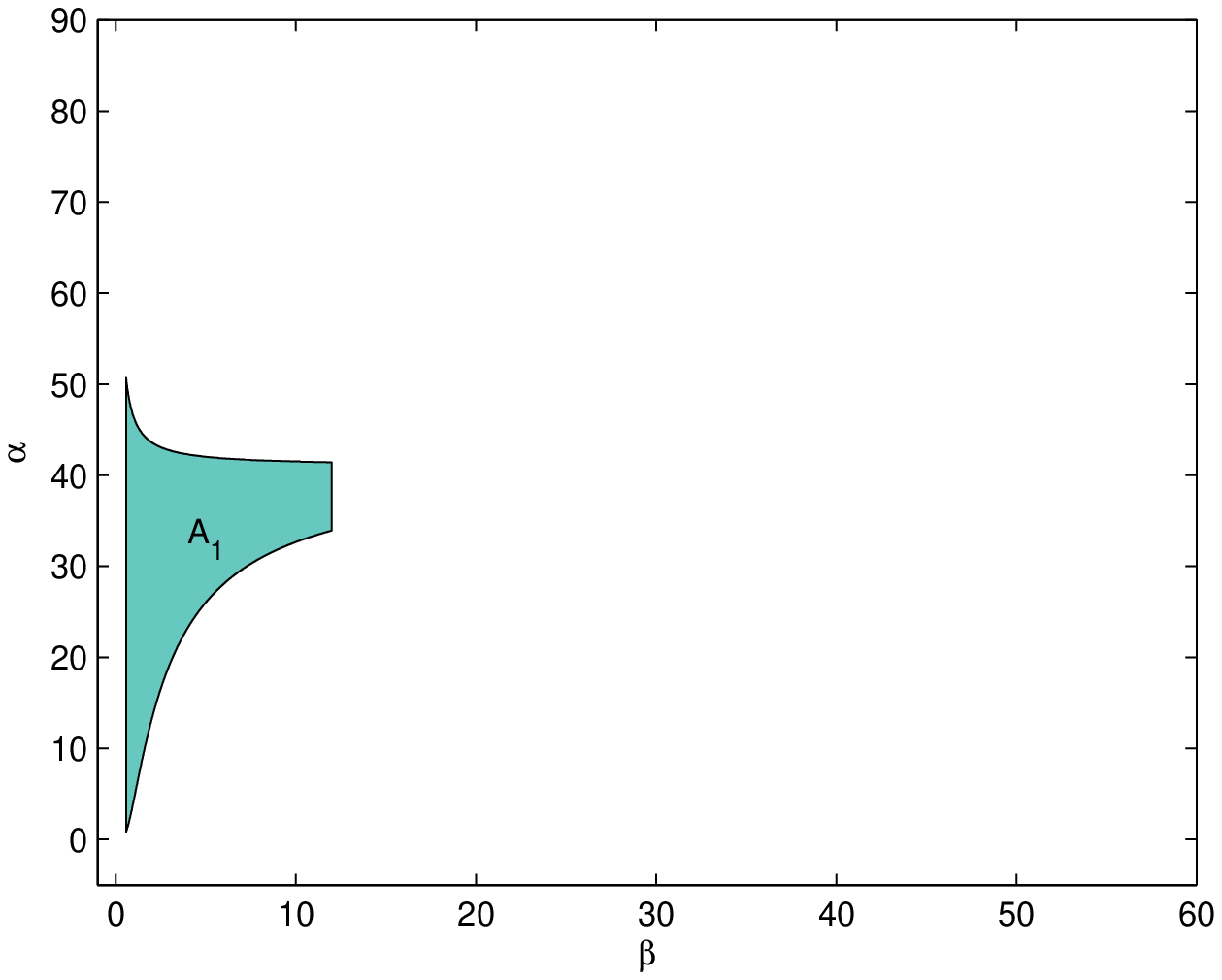}
\includegraphics[scale=0.5]{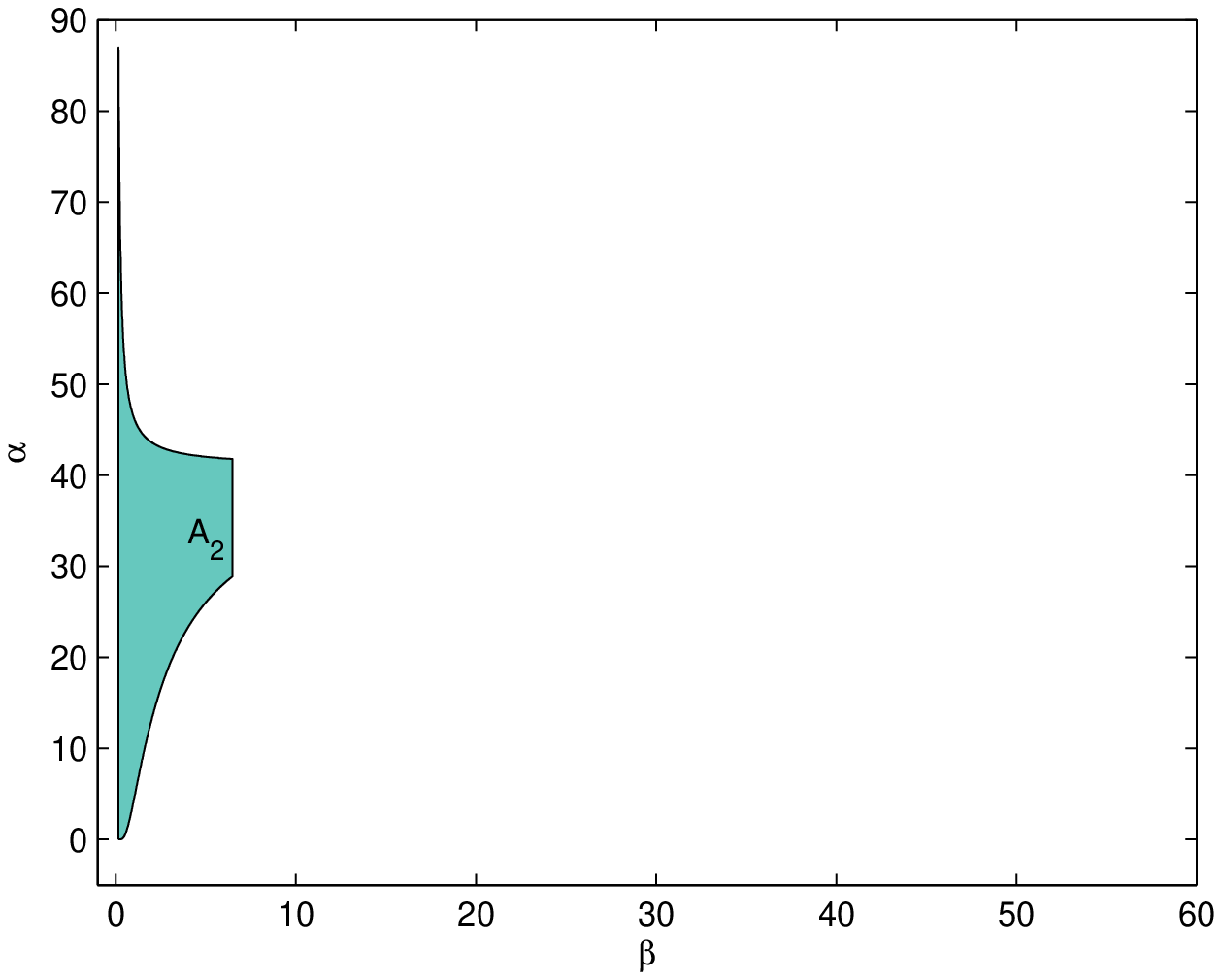}
\includegraphics[scale=0.5]{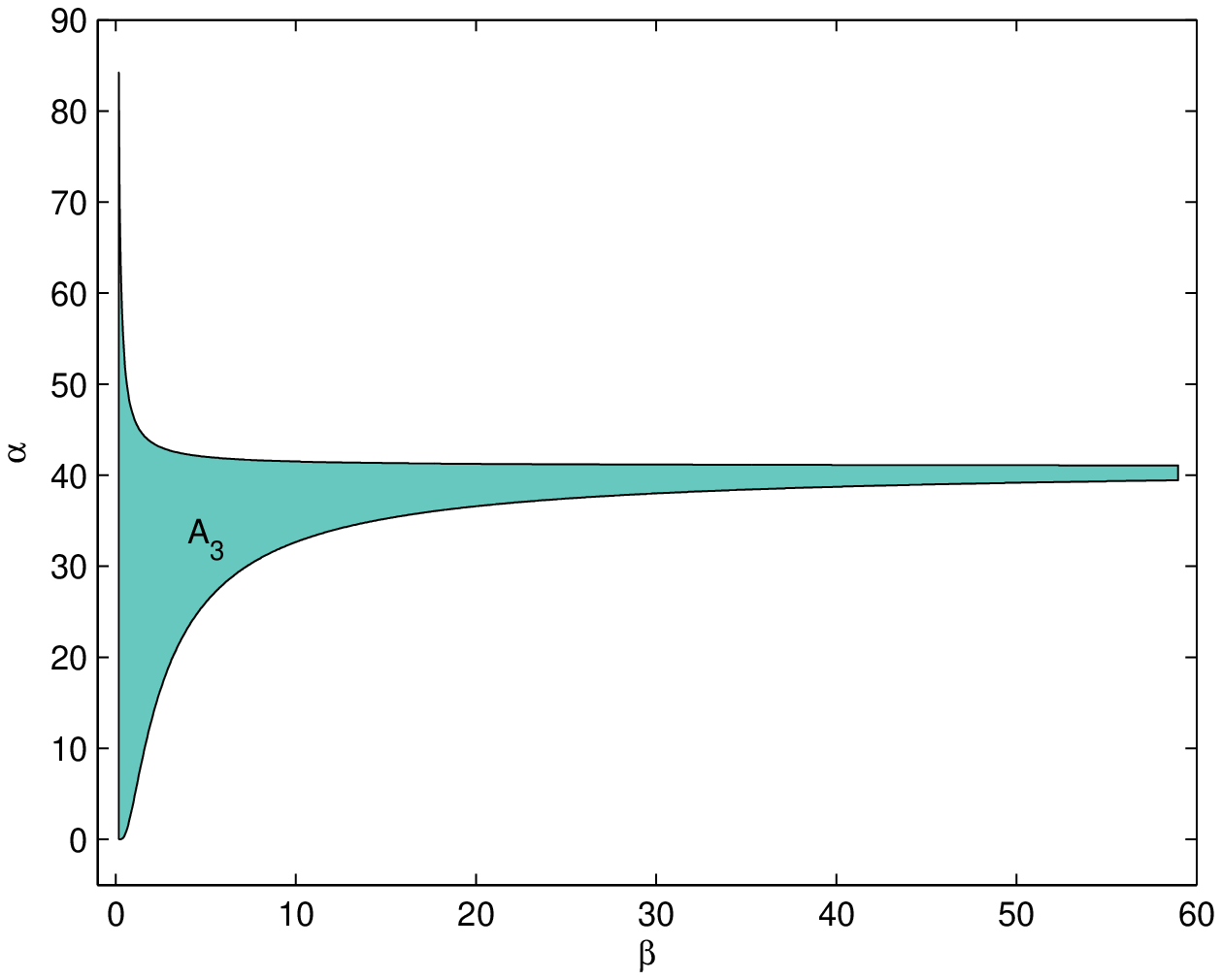}
\includegraphics[scale=0.5]{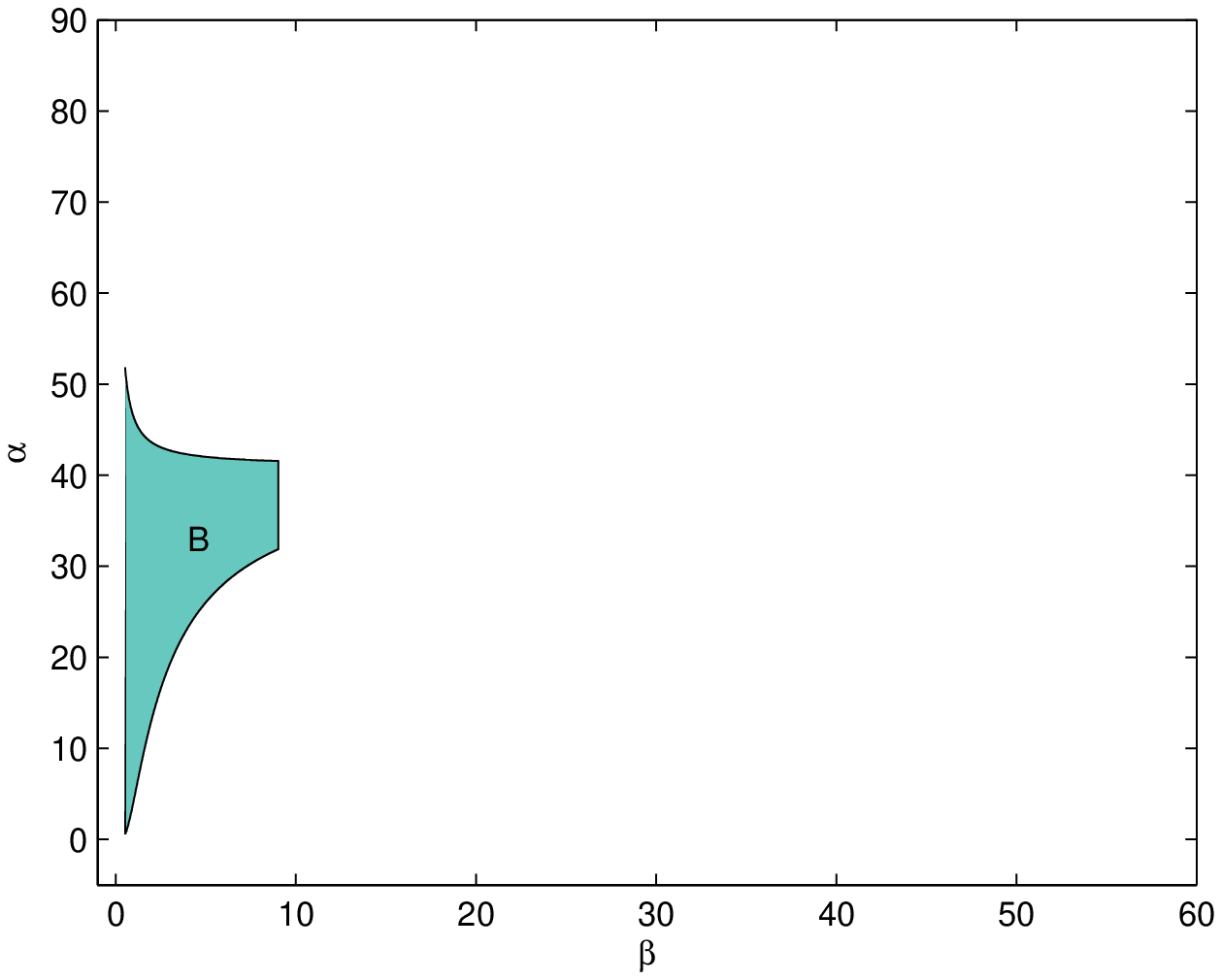}
\vspace{-0.4cm}
\caption[]{The plot of the joint confidence regions $A_1$, $A_2$, $A_3$, and $B$. } \label{fig.joint}
\end{figure}

\newpage

\begin{table}
\begin{center}
\caption{ Empirical coverage probabilities of the methods for constructing joint confidence region for the parameters $\alpha $ and $\beta $ with $\gamma =0.05$.} \label{tab.ecjoint}

\begin{tabular}{|c|c|c|ccccccc|} \hline
& &  &\multicolumn{7}{|c|}{$\beta $}      \\ \hline
& $n$ & Region & 0.5 & 1.0 & 1.2 & 1.5 & 2.0 & 3.0 & 5.0 \\ \hline
Coverage &4 & $A_1$ & 0.950 & 0.949 & 0.949 & 0.951 & 0.954 & 0.946 & 0.950 \\
Probability&  & $A_2$ & 0.949 & 0.951 & 0.950 & 0.951 & 0.953 & 0.949 & 0.952 \\
 && $B$  & 0.949 & 0.950 & 0.949 & 0.952 & 0.954 & 0.949 & 0.950 \\ \cline{2-10}
&6 & $A_1$ & 0.951 & 0.951 & 0.952 & 0.952 & 0.954 & 0.949 & 0.949 \\
& & $A_2$ & 0.952 & 0.948 & 0.950 & 0.948 & 0.953 & 0.948 & 0.952 \\
& & $B$ & 0.953 & 0.949 & 0.951 & 0.951 & 0.953 & 0.950 & 0.951 \\ \cline{2-10}
&9 
 & $A_2$ & 0.953 & 0.949 & 0.950 & 0.955 & 0.949 & 0.948 & 0.953 \\
& & $A_3$ & 0.951 & 0.951 & 0.951 & 0.956 & 0.949 & 0.948 & 0.949 \\
& & $B$ & 0.950 & 0.952 & 0.951 & 0.953 & 0.949 & 0.948 & 0.952 \\ \cline{2-10}
&14
 & $A_3$ & 0.947 & 0.950 & 0.954 & 0.948 & 0.954 & 0.951 & 0.952 \\
& & $A_4$ & 0.946 & 0.951 & 0.952 & 0.948 & 0.951 & 0.950 & 0.952 \\
& & $B$ & 0.948 & 0.949 & 0.952 & 0.947 & 0.951 & 0.953 & 0.952 \\\cline{2-10}
& 29& $A_6$ & 0.951 & 0.953 & 0.950 & 0.950 & 0.948 & 0.948 & 0.950 \\
&   & $A_7$ &0.950 & 0.952 & 0.951 & 0.949 & 0.948 & 0.948 & 0.950 \\
&   & $B$ &0.953 & 0.955 & 0.951 & 0.953 & 0.950 & 0.948 & 0.952 \\ \hline \hline
%
%
Expected &4 & $A_1$ & 27.787 & 8.548 & 7.339 & 6.331 & 5.725 & 5.330 & 5.203 \\
Area& & $A_2$ & 30.020 & 8.976 & 7.651 & 6.682 & 5.989 & 5.593 & 5.504 \\
& & $B$  & 22.985 & 7.371 & 6.388 & 5.596 & 5.099 & 4.792 & 4.713 \\ \cline{2-10}
&6 & $A_1$ & 21.062 & 6.036 & 5.213 & 4.576 & 4.102 & 3.783 & 3.701 \\
& & $A_2$ & 20.035 & 5.824 & 5.046 & 4.436 & 3.985 & 3.701 & 3.648 \\
& & $B$ & 14.551 & 4.714 & 4.144 & 3.714 & 3.399 & 3.192 & 3.162 \\ \cline{2-10}
&9
 & $A_2$ & 14.631 & 4.081 & 3.533 & 3.059 & 2.756 & 2.574 & 2.510 \\
& & $A_3$ & 14.651 & 4.086 & 3.545 & 3.077 & 2.767 & 2.585 & 2.541 \\
& & $B$ & 9.639 & 3.137 & 2.774 & 2.471 & 2.275 & 2.160 & 2.133 \\ \cline{2-10}
&14
 & $A_3$ & 9.436 & 2.702 & 2.320 & 2.035 & 1.816 & 1.701 & 1.661 \\
& & $A_4$ & 9.388 & 2.686 & 2.304 & 2.035 & 1.812 & 1.707 & 1.668 \\
&    & $B$   & 5.784 & 1.999 & 1.763 & 1.599 & 1.471 & 1.405 & 1.385 \\ \cline{2-10}
& 29 & $A_6$ & 4.244 & 1.298 & 1.124 & 0.988 & 0.905 & 0.849 & 0.828 \\
&    & $A_7$ &4.202 & 1.291 & 1.118 & 0.985 & 0.904 & 0.850 & 0.828 \\
&    & $B$   &2.380 & 0.932 & 0.838 & 0.761 & 0.719 & 0.689 & 0.678 \\ \hline
\end{tabular}
\end{center}
\end{table}

\subsection{Simulation study}

We performed a similar simulation given in Section \ref{sec.sim1} with  considering $\alpha =1$, in order to compare the joint confidence regions proposed by
\cite{as-ab-11}
and our joint confidence region ($B$) in \eqref{joint.B}.
Here, we consider the confidence areas $A_{\lfloor \frac{n+1}{5}\rfloor}$ and $A_{{\lfloor \frac{n+1}{5}+1\rfloor}}$ because the coverage probabilities of all $A_i$'s are close to the confidence coefficient and \cite{as-ab-11} found that in most cases these two confidence areas provide the smallest confidence areas.
The empirical coverage probabilities and expected areas of the methods for the confidence coefficient 95\% are given in Table \ref{tab.ecjoint}. We can conclude that

\begin{enumerate}
  \item The coverage probabilities of the all methods are close to the confidence coefficient 0.95.
  \item The expected area of our method is smaller than the expected areas of the proposed methods by
  \cite{as-ab-11}.
  \item The expected areas of all methods decrease when the number of records increases.
  \item The expected areas of all methods decrease when the parameter $\beta $ increases.
\end{enumerate}

\section*{Appendix. Generalized {\it p}-value and generalized confidence interval}

Let ${\boldsymbol X}$ be a random variable whose distribution depends on a parameter of interest $\theta $, and a nuisance parameter $\lambda $. Let ${\boldsymbol x}$ denote the observed value of ${\boldsymbol X}$. A generalized pivotal quantity for $\theta $ is a random quantity denoted by $T({\boldsymbol X};{\boldsymbol x};\theta )$ that satisfies the following conditions:

\noindent (i) The distribution of $T({\boldsymbol X};{\boldsymbol x};\theta )$ is free of any unknown parameters.

\noindent (ii) The value of $T({\boldsymbol X};{\boldsymbol x};\theta )$ at ${\boldsymbol X}\ =\ {\boldsymbol x}$, i.e., $T({\boldsymbol x};{\boldsymbol x};\theta)$ is free of the nuisance parameter $\lambda $.

Appropriate percentiles of $T({\boldsymbol X};{\boldsymbol x};\theta )$ form a confidence interval for $\theta $. Specifically, if $T({\boldsymbol x};{\boldsymbol x};\theta)=\theta $, and  $T_{\delta }$ denotes the 100$\delta $ percentage point of $T({\boldsymbol X};{\boldsymbol x};\theta )$ then $(T_{\gamma /2},\ T_{1-\gamma /2})$ is a $1-\gamma $ generalized confidence interval for $\theta $. The percentiles can be found because, for a given ${\boldsymbol x}$, the distribution of $T({\boldsymbol X};{\boldsymbol x};\theta )$ does not depend on any unknown parameters.

 In the above setup, suppose we are interested in testing the hypotheses
\begin{equation}\label{H0.g}
H_0:\theta \le {\theta }_0\ \ \ \ \ \ \ \ vs.\ \ \ \ \ \ \ H_1:\theta >{\theta }_0, \tag{A.1}
\end{equation}
 for a specified  ${\theta }_0$. The generalized test variable, denoted by $T^*({\boldsymbol X};{\boldsymbol x};\theta )$, is defined as follows:

\noindent (i) The value of $T^*({\boldsymbol X};{\boldsymbol x};\theta )$ at ${\boldsymbol X}\ =\ {\boldsymbol x}$\textit{ }is free of any unknown parameters.

\noindent (ii) The distribution of $T^*({\boldsymbol X};{\boldsymbol x};\theta )$ is stochastically monotone (i.e., stochastically increasing or stochastically decreasing) in $\theta $ for any fixed ${\boldsymbol x}$ and  $\lambda$.

\noindent (iii) The distribution of $T^*({\boldsymbol X};{\boldsymbol x};\theta )$ is free of any unknown parameters.

Let $t^*=T^*\left({\boldsymbol x};{\boldsymbol x};{\theta }_0\right)$, the observed value of $T^*({\boldsymbol X};{\boldsymbol x};\theta )$ at $\left({\boldsymbol X};\theta \right)=({\boldsymbol x};{\theta }_0)$. When the above conditions hold, the generalized \textit{p}-value for testing the hypotheses in \eqref{H0.g} is defined as
\begin{equation}\label{pv}
 p = P\left(T^*\left({\boldsymbol X};{\boldsymbol x};{\theta }_0\right)\le t^*\right)
\tag{A.2}
\end{equation}
where $T^*({\boldsymbol X};{\boldsymbol x};\theta )$ is stochastically decreasing in $\ \theta $. The test based on the generalized \textit{p}-value rejects $H_0$ when the generalized \textit{p}-value is smaller than a nominal level $\gamma $. However, the size and power of such a test may depend on the nuisance parameters.

\section*{Acknowledgements}
The authors would like to thank  two referees for their helpful comments and suggestions which have contributed to improving the manuscript.

\bibliographystyle{apa}

\end{document}